\documentclass[11pt]{article}
\usepackage{amsmath,amsthm,amssymb,amsfonts}
\usepackage{multicol, latexsym}
\usepackage{latexsym}
\usepackage{graphicx,psfrag,import}
\usepackage{fullpage}
\usepackage{framed}
\usepackage{verbatim}
\usepackage{color}
\usepackage{epsfig}
\usepackage[outdir=./]{epstopdf}
\usepackage{hyperref}   
\usepackage{geometry}
\usepackage{mathtools}
\usepackage{nonfloat}
\usepackage{enumerate}
\usepackage{enumitem}   
\usepackage{multicol}
\usepackage{a4wide}
\usepackage{booktabs}
\usepackage{enumitem}
\usepackage{lineno}
\usepackage{parcolumns}
\usepackage{thmtools}
\usepackage{xr}
\usepackage{epstopdf}
\usepackage{mathrsfs}
\usepackage{caption}
\usepackage{subcaption}

\usepackage{comment}
\usepackage{authblk}
\usepackage{setspace}

\mathchardef\mhyphen="2D

\theoremstyle{plain}
\newtheorem{theorem}{Theorem}[section]

\newtheorem{proposition}[theorem]{Proposition}
\newtheorem{lemma}[theorem]{Lemma}

\theoremstyle{definition}
\newtheorem{definition}[theorem]{Definition}

\newtheorem{remark}[theorem]{Remark}

\theoremstyle{remark}

\newcommand\vx{\boldsymbol{x}}
\newcommand\vy{\boldsymbol{y}}

\newcommand\by{\boldsymbol{y}}

\newcommand\vv{\boldsymbol{v}}

\newcommand\s{\boldsymbol{s}}
\newcommand\z{\boldsymbol{z}}

\newcommand\GG{\mathcal{G}}

        {\begin{list}
                {\noindent\makebox[0mm][r]{$\bullet$}}
                {\leftmargin=5.5ex \usecounter{enumi}
      \topsep=1.5mm \itemsep=-.75ex}
        }
        {\end{list}}
        
\begin{document}

\title{Minimal invariant regions and minimal globally attracting regions for variable-$k$ reaction systems}

\author[1]{Yida Ding}
\author[2]{Abhishek Deshpande}
\author[3]{Gheorghe Craciun}
\affil[1]{Department of Mathematics, University of Wisconsin-Madison, {\tt yding54@wisc.edu}.}
\affil[2]{Center for Computational Natural Sciences and Bioinformatics, International Institute of Information Technology Hyderabad, {\tt abhishek.deshpande@iiit.ac.in}.}
\affil[3]{Department of Mathematics and Department of Biomolecular Chemistry, University of Wisconsin-Madison, {\tt craciun@math.wisc.edu}.}

\maketitle

\begin{abstract}
The structure of invariant regions and globally attracting regions is fundamental to understanding the dynamical properties of reaction network models. We  describe an explicit  construction of the minimal invariant regions and minimal globally attracting regions for dynamical systems consisting of two reversible reactions, where the rate constants are allowed to vary in time within a bounded interval. 
\end{abstract}

\section{Introduction}

Reaction networks are ubiquitous in several mathematical models arising in biology, physics and chemistry. These models often incorporate differential equations with polynomial or power-law right hand sides~\cite{savageau1969biochemical} of the form given by
\begin{equation}\label{eq:power_law_dyn_system}
\frac{d\vx}{dt}=\sum\limits_{i=1}^m  k_i {\vx}^{\s_i} \vv_i
\end{equation}
where $\vx = (x_1, x_2, ..., x_n) \in \mathbb{R}^n_{>0}$, $\s_i, \vv_i\in\mathbb{R}^n$, and $\vx^{\vy }:=x_1^{y_{1}}x_2^{y_{2}}...x_n^{y_{n}}$.

Associated with such dynamical systems is a property called \emph{persistence} which implies that no species can go extinct, i.e., $\displaystyle\liminf x_i(t)>0$ for all $i$. The property of persistence is related to some of the most important open problems in reaction network theory, such as  the \emph{Persistence Conjecture} and the \emph{Global Attractor Conjecture}. 
Several special cases of these conjectures have been proved in the last few years, but many important problems are still open~\cite{anderson2011proof,craciun2013persistence,gopalkrishnan2014geometric,pantea2012persistence}. It is therefore important to analyze invariant regions and globally attracting regions for these systems.

In general rate constants associated with reactions can vary in some range due to the change in environment like the change in pressure, temperature or external signals, etc. Mathematically, this means that the rate constants $k_i(t)$ are functions of time. In particular, if we set constraints on the rate constants to lie in a bounded interval, these systems are called variable-$k$ dynamical systems. For example, consider the following network
\begin{eqnarray}\label{ex:network}
\begin{aligned}
 Y &\xrightleftharpoons[k_2(t)]{k_1(t)} 2 X \\
 X &\xrightleftharpoons[k_4(t)]{k_3(t)} 2 Y 
\end{aligned}
\end{eqnarray}
If the rate constants satisfy $\epsilon\leq k_1(t), k_2(t), k_3(t), k_4(t)\leq \frac{1}{\epsilon}$, then the dynamical systems generated by networks like~\ref{ex:network} are called variable-$k$ dynamical systems. In this paper, we give an explicit construction of the minimal invariant regions and minimal globally attracting regions for variable-$k$ dynamical systems generated by two dimensional reversible reaction networks similar to network~\ref{ex:network} described above.

This paper is structured as follows: In Section~\ref{sec:persistence_permanence}, we introduce the notions of persistence, permanence and uncertainty regions. In Section~\ref{sec:invariant_globally}, we formally define the notions of minimal invariant regions and the minimal globally attracting regions for our dynamical systems. In Section~\ref{min_invariant_globally}, we give an explicit construction of the minimal invariant region and the minimal globally attracting region for two dimensional variable-$k$ dynamical systems.

\section{E-graphs, Persistence and Permanence}\label{sec:persistence_permanence}

A {\em reaction network} is a  directed graph $\GG=(V,E)$, with a finite set of vertices $V \subset \mathbb R^n$ and a set of edges $E \subset V \times V$. Such a graph is also called  an {\em Euclidean embedded graph} (or {\em E-graph})~\cite{craciun2019polynomial}. If there is an edge from the vertex $\by$ to the vertex $\by'$ in an E-graph, we will also denote this by the {\em reaction} $\by\rightarrow \by'$ (i.e, the {\em reactions} are just the edges of $\GG$). We will say that an E-graph $\GG=(V,E)$ is \emph{reversible} if for every edge $\by\rightarrow \by'$, there exists an edge $\by'\rightarrow \by$. We will say that an E-graph is \emph{weakly reversible} if every edge is part of a cycle. The {\em reaction vector} of a reaction $\by\rightarrow \by'$ is the vector $\by' - \by$. The span of the reaction vectors is the  \emph{stoichiometric subspace} $S$ of $\GG$, i.e., it is given by $S=\{\rm{span}(\by'-\by)\,|\,\by\rightarrow \by'\in E\}$. If we fix some  $\vx_0\in\mathbb{R}^n_{>0}$ then the \emph{stoichiometric compatibility class} (denoted by $\mathcal{C}$) corresponding to $\vx_0$ is given by $\mathcal{C}=(\vx_0 + S)\cap \mathbb{R}^n_{>0}$.

Every reaction network generates a family of dynamical systems on the positive orthant. If we assume mass-action kinetics~\cite{adleman2014mathematics,voit2015150,waage1986studies,gunawardena2003chemical, yu2018mathematical}, the dynamical systems generated by a reaction network are given by
\begin{eqnarray}\label{eq:mass_acction}
\frac{d\vx}{dt}=\sum\limits_{\vy\rightarrow \vy ' \in E} k_{\vy\rightarrow \vy '} {\vx}^{\vy} (\vy ' - \vy) 
\end{eqnarray}
where $k_{\vy\rightarrow \vy '} > 0$ is the {\em rate constant} of the reaction $\vy\rightarrow \vy '$. In general, rate constants can be time-dependent to accommodate the uncertainty introduced by external influences. In this case, the reaction network generates \emph{non-autonomous} dynamical systems given by 
\begin{eqnarray}\label{eq:non_autonomous}
\frac{d\vx}{dt}=\sum\limits_{\vy\rightarrow \vy ' \in E} k_{\vy\rightarrow \vy '}(t) {\vx}^{\vy} (\vy ' - \vy) 
\end{eqnarray}
In particular, if the rate constants corresponding to the reactions are allowed to take values in the bounded interval $[\epsilon,\frac{1}{\epsilon}]$ for some $\epsilon>0$, then the dynamical systems they generate are called \emph{variable-$k$ dynamical systems}~\cite{craciun2013persistence,craciun2015toric}.

We now define some important dynamical properties of reaction networks.

\begin{definition}[Persistence]\label{defn:persistent}
Consider a dynamical system of the form~(\ref{eq:non_autonomous}). This dynamical system is said to be \emph{persistent} if for any initial condition $\vx(0)\in\mathbb{R}^n_{>0}$, the solution $\vx(t)$ of~(\ref{eq:non_autonomous}) satisfies 
\begin{eqnarray}
\displaystyle\liminf_{t\rightarrow T_{\rm max}}\vx_i(t)>0 
\end{eqnarray}
for all $i=1,2,...,n$ where $T_{\rm max}$ is the maximal time for which the solution exists.
\end{definition}

\begin{definition}[Permanence]\label{defn:permanent}
Consider a dynamical system of the form~(\ref{eq:non_autonomous}). This dynamical system is said to be \emph{permanent} if for each stoichiometric compatibility class $\mathcal{C}$, there exists a compact set $\mathcal{D}\subseteq \mathcal{C}$ and a time $T>0$ such that for any solution $\vx(t)$ of~(\ref{eq:non_autonomous}) with $\vx(0)\in\mathcal{C}$, we have $\vx(t)\in \mathcal{D}$ for all $t\geq T$.
\end{definition}

\begin{definition}[Detailed balance]\label{defn:detailed_balance}
Consider a dynamical system of the form~(\ref{eq:non_autonomous}) generated by a reversible E-graph $\GG=(V,E)$. This dynamical system is said to be \emph{detailed balanced} if there exists ${\vx}_0\in\mathbb{R}^n_{>0}$ such that the following holds for every reversible reaction $\vy\rightleftharpoons \vy' \in E$:
\begin{eqnarray}
\displaystyle\sum_{\vy\rightarrow \vy' \in E}k_{\vy\rightarrow \vy'}{\vx}_0^{\vy} = \displaystyle\sum_{\vy'\rightarrow \vy\in E}k_{\vy'\rightarrow \vy}{\vx}_0^{\vy'}.
\end{eqnarray}
\end{definition}

\begin{definition}[Complex balance]\label{defn:complex_balance}
Consider a dynamical system of the form~(\ref{eq:non_autonomous}) generated by an E-graph $\GG=(V,E)$. This dynamical system is said to be \emph{complex balanced} if there exists ${\vx}_0\in\mathbb{R}^n_{>0}$ such that the following holds for every vertex $\vy\in V$:
\begin{eqnarray}
\displaystyle\sum_{\vy\rightarrow \vy' \in E}k_{\vy\rightarrow \vy'}{\vx}_0^{\vy} = \displaystyle\sum_{\vy'\rightarrow \vy\in E}k_{\vy'\rightarrow \vy}{\vx}_0^{\vy'}.
\end{eqnarray}
\end{definition}

\bigskip

\noindent
In what follows, we state some of the most important open problems in reaction network theory~\cite{craciun2013persistence}:

\begin{enumerate}
\item \textbf{Persistence conjecture:} Any dynamical system generated by a weakly reversible E-graph is persistent.
\item \textbf{Extended Persistence conjecture:} Any variable-$k$ dynamical system generated by an endotactic E-graph is persistent.
\item \textbf{Permanence conjecture:} Any dynamical system generated by a weakly reversible E-graph is permanent.
\item \textbf{Extended Permanence conjecture:} Any variable-$k$ dynamical system generated by an endotactic E-graph is permanent.
\end{enumerate}

The above conjectures are very closely related to the \emph{Global Attractor Conjecture}, which states that complex balanced dynamical systems have a globally attracting fixed point~\cite{craciun2009toric}.  In particular, the proof of any one of these four conjectures would also imply a proof of the Global Attractor Conjecture~\cite{craciun2009toric,siegel2000global,sontag2001structure}. Several special cases of these conjectures have been proved. Craciun, Nazarov and Pantea~\cite{craciun2013persistence} have proved the extended permanence conjecture in two dimensions. This has been extended by Pantea~\cite{pantea2012persistence} to the case of E-graphs with two dimensional stoichiometric subspace. Anderson~\cite{anderson2011proof} has proved the Global Attractor Conjecture for E-graphs consisting of a single connected component. Gopalkrishnan, Miller and Shiu~\cite{gopalkrishnan2014geometric} have shown that variable-$k$ dynamical systems generated by strongly endotactic E-graphs are permanent. A fully general proof of the Global Attractor Conjecture has recently been proposed by Craciun~\cite{craciun2015toric}. An essential component of the proof relies on building invariant regions for certain dynamical systems. In this paper, we give an explicit construction of the \emph{minimal invariant regions} and \emph{minimal globally attracting regions} for variable-$k$ dynamical systems generated by two reversible reactions.

\begin{definition}[Uncertainty region]
Consider the reversible reaction
\begin{eqnarray}\label{eq:example_reaction}
\begin{aligned}
a X + b Y &\xrightleftharpoons[k_2(t)]{k_1(t)} a'X + b'Y 
\end{aligned}
\end{eqnarray}
Now consider the variable-$k$ dynamical systems generated by this reversible reaction if we choose rate constants as follows:
\begin{enumerate}
\item $k_1(t)=\epsilon$ and $k_2(t)=\frac{1}{\epsilon}$.
\item $k_1(t)=\frac{1}{\epsilon}$ and $k_2(t)=\epsilon$.
\end{enumerate}
The condition for these dynamical systems to be detailed balanced gives the curves $y^{b'-b}=\epsilon^2x^{a-a'}$ and 
$y^{b'-b}=\frac{1}{\epsilon^2}x^{a-a'}$ respectively. The \emph{uncertainty region} corresponding to the reaction~(\ref{eq:example_reaction}) is the region enclosed between the curves $y^{b'-b}=\epsilon^2x^{a-a'}$ and 
$y^{b'-b}=\frac{1}{\epsilon^2}x^{a-a'}$.
\end{definition}

\begin{definition}[Attracting directions of an uncertainty region]
Consider an uncertainty region corresponding to the reaction given by
\begin{eqnarray}\label{eq:attracting_reaction}
\begin{aligned}
a X + b Y &\xrightleftharpoons[k_2(t)]{k_1(t)} a'X + b'Y 
\end{aligned}
\end{eqnarray}
Note that this uncertainty region divides the positive orthant into three connected components, as shown in Figure~\ref{fig:attr_dir}. For components that lie outisde the uncertainty region, the attracting direction is the direction perpendicular to the line $(b'-b)y=(a'-a)x$ and points towards this uncertainty region. Within the uncertainty region, the attracting direction is perpendicular to the line $(b'-b)y=(a'-a)x$ (i.e. parallel to the reaction vector). The exact direction will be determined by the choice of rate constants in Equation~(\ref{eq:attracting_reaction}). 

\end{definition}

\begin{figure}[h!]
\centering
\includegraphics[scale=0.55]{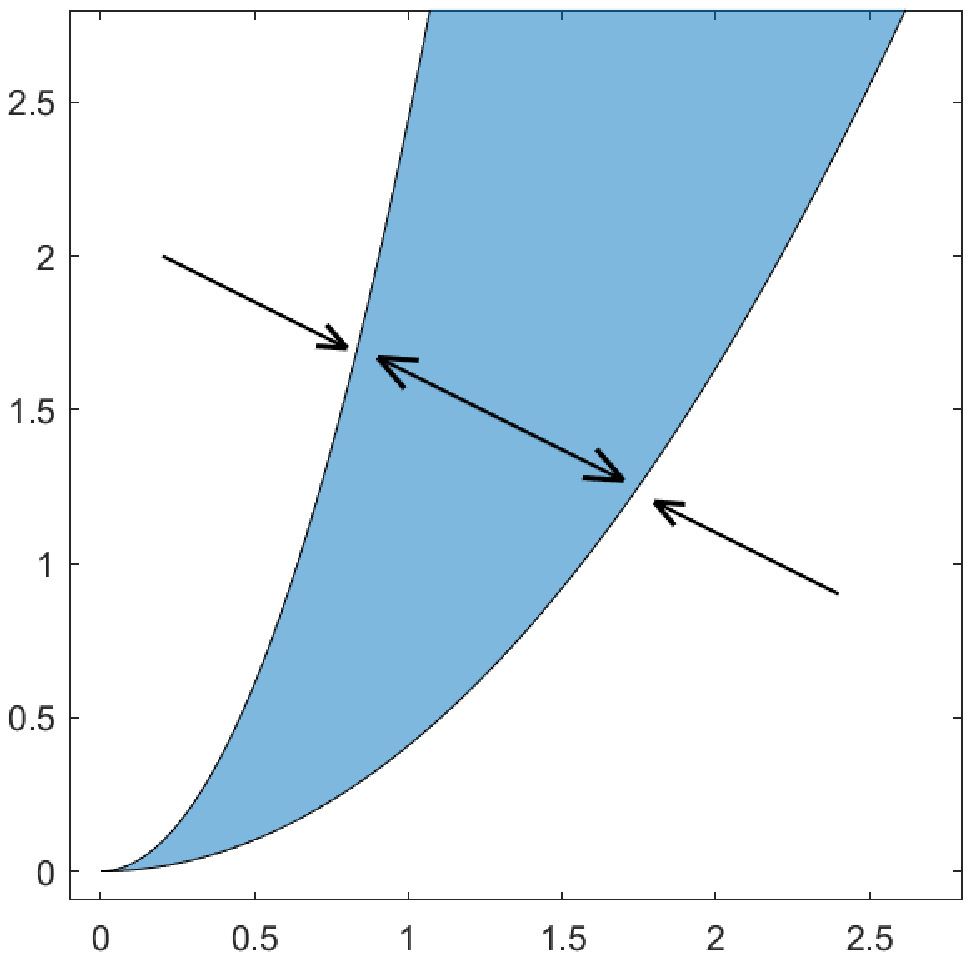}
\includegraphics[scale=0.55]{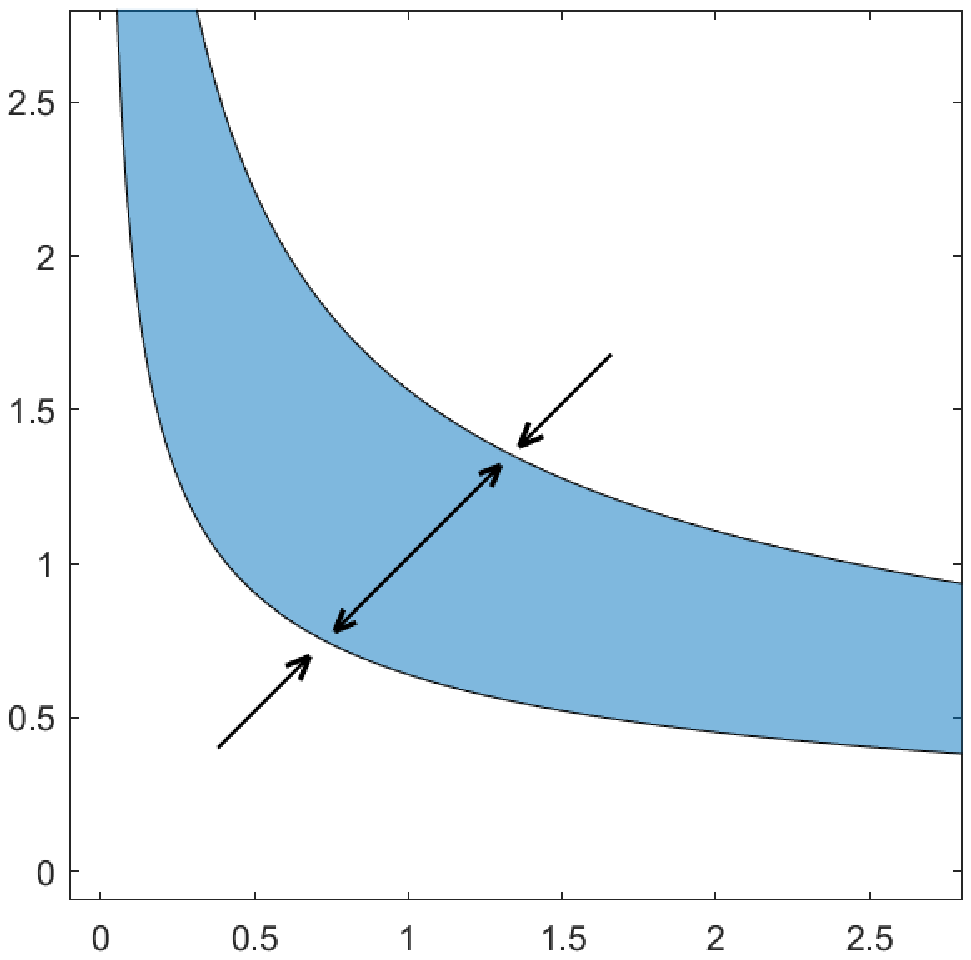}
\caption{Attracting directions corresponding to the uncertainty region of a reversible reaction. On the left we show the case where the reaction has {\em negative} slope, and on the right we show the case where the reaction has {\em positive} slope.}
\label{fig:attr_dir}
\end{figure}

\begin{remark}
It is easy to see in Figure~\ref{fig:attr_dir} that, for networks that consist of a \underline{single} reversible reaction, the blue regions are {\em minimal globally attracting regions}. Moreover, any (line segment) obtained as the intersection between a blue region and a stoichiometric compatibility class is a {\em minimal invariant region}. Our goal in this paper is {\em to solve this problem for the simplest nontrivial case: the case of \underline{two} reversible reactions.}
\end{remark}

\section{Variable\mbox{-}k reaction systems given by two reversible reactions}\label{sec:invariant_globally}

Considering the following reaction network:

\begin{eqnarray}\label{eq:variable-k}
\begin{aligned}
a_1 X + b_1 Y &\xrightleftharpoons[k_2(t)]{k_1(t)} a'_1 X + b'_1 Y \\
a_2 X + b_2 Y &\xrightleftharpoons[k_4(t)]{k_3(t)} a'_2 X + b'_2 Y 
\end{aligned}
\end{eqnarray}

\begin{figure}[h!]
\centering
\includegraphics[scale=0.6]{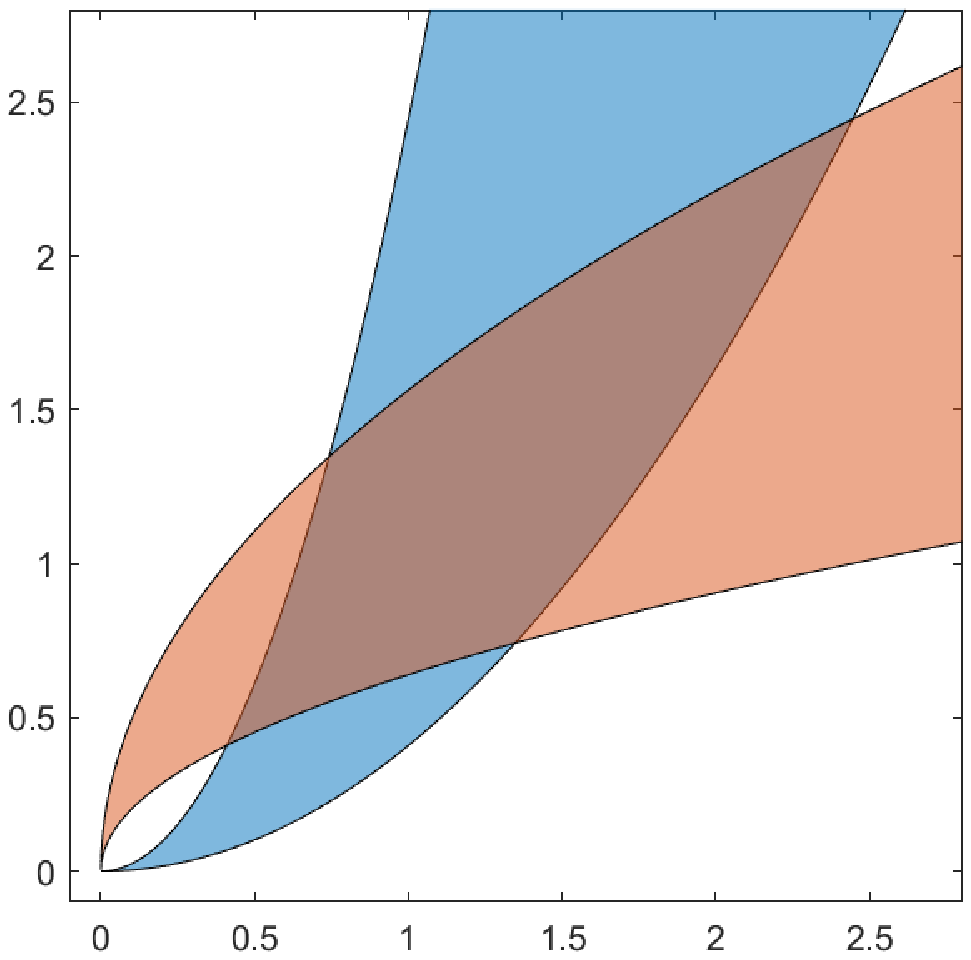}
\caption{An example of uncertainty regions corresponding to the reaction network described in Equation~(\ref{eq:variable-k}). Here we illustrate the case where the slopes of both of the reaction vectors are negative.}
\label{fig:two_curves}
\end{figure}

\begin{remark}
Without loss of generality, we will denote the uncertainty regions corresponding to the reactions $a_1 X + b_1 Y \xrightleftharpoons[k_2(t)]{k_1(t)} a'_1 X + b'_1 Y$ and $a_1 X + b_1 Y \xrightleftharpoons[k_2(t)]{k_1(t)} a'_1 X + b'_1 Y$ by the red and blue regions respectively as shown in Figure~\ref{fig:two_curves}.
\end{remark}

In what follows, we will denote the variable-$k$ dynamical system generated by Equation~(\ref{eq:variable-k}) by $\GG_{\epsilon}^{\rm{variable\mbox{-}k}}$.

\begin{definition}
Let $\vx(t)$ be a solution of $\GG_{\epsilon}^{\rm{variable\mbox{-}k}}$ with initial condition $\vx(0)\in\mathbb{R}^n_{>0}$. The \emph{omega-limit} of this solution is the set $\omega(\vx(0))=\{\z\in\mathbb{R}^n_{\geq 0}: \text{there is an increasing sequence of times} \ (t_h)_{h\in\mathbb{N}}\\ \text{such that} \displaystyle\lim_{h\to\infty}t_h=\infty$ and $\displaystyle\lim_{h\to\infty}\vx(t_h)=\z\}$. 
\end{definition}

\begin{definition}
A set $\Omega^{\rm inv}_{\rm{\rm{variable\mbox{-}k}}}\in\mathbb{R}^n_{>0}$ is said to be a \textit{closed invariant region} if it is closed and for any solution $\vx(t)$ of $\GG_{\epsilon}^{\rm{variable\mbox{-}k}}$ with $\vx(0)\in\Omega^{\rm inv}_{\rm{\rm{variable\mbox{-}k}}}$, we have $\vx(t)\in\Omega^{\rm inv}_{\rm{\rm{variable\mbox{-}k}}}$ for all $t>0$. A set $\Omega^{\min,\rm inv}_{\rm{\rm{variable\mbox{-}k}}}$ is said to be the \textit{minimal closed invariant region} if for any closed invariant region $\Omega_{\rm{\rm{variable\mbox{-}k}}}$, we have $\Omega^{\min, \rm inv}_{\rm{\rm{variable\mbox{-}k}}}\subseteq\Omega_{\rm{variable\mbox{-}k}}$. For simplicity, instead of \textit{minimal closed invariant region} we will simply say  \textit{minimal  invariant region}.
\end{definition}

\begin{definition}
A set $\Omega^{\rm glob}_{\rm{variable\mbox{-}k}}\in\mathbb{R}^n_{>0}$ is said to be a \textit{globally attracting region} if for any solution $\vx(t)$ of $\GG_{\epsilon}^{\rm{variable\mbox{-}k}}$ with $\vx(0)\in\mathbb{R}^n_{>0}$, we have $\omega(\vx(0))\subseteq \Omega^{\rm glob}_{\rm{variable\mbox{-}k}}$. A set $\Omega^{\min,\rm glob}_{\rm{variable\mbox{-}k}}$ is said to be the \textit{minimal globally attracting region} if for any globally attracting region $\Omega_{\rm{variable\mbox{-}k}}$, we have $\Omega^{\min, \rm glob}_{\rm{variable\mbox{-}k}}\subseteq\Omega^{\rm glob}_{\rm{variable\mbox{-}k}}$. 
\end{definition}

\begin{definition}
Given two points $P,Q\in\mathbb{R}^n_{>0}$, we say $P\leadsto Q$ if there is a solution $\vx(t)$ of $\GG_{\epsilon}^{\rm{variable\mbox{-}k}}$ such that $\vx(0) = P$ and for every $\xi>0$ there exists a $T$ that satisfies $||x(T)-Q||\leq\xi$.   
\end{definition}

\section{Minimal invariant region and minimal globally attracting region for variable-$k$ dynamical systems generated by two reversible reactions}\label{min_invariant_globally}

According to how we choose the parameters $a_1,a'_1,b_1,b'_1,a_2,a'_2,b_2,b'_2$, we get the following cases that correspond to various orientations of the uncertainty regions.
\begin{enumerate}[label=(\roman*)]
\item \emph{Both reaction vectors with negative slopes; one with slope less than $-1$ and the other with slope greater than $-1$}: $-1<\frac{b'_1 - b_1}{a'_1-a_1}<0$ \text{and} $\frac{b'_2 - b_2}{a'_2-a_2}< -1$.
\item \emph{Both reaction vectors with slopes between $0$ and $-1$}: $-1<\frac{b'_1 - b_1}{a'_1-a_1},\frac{b'_2 - b_2}{a'_2-a_2} <0$.
\item \emph{Both reaction vectors with slopes less than $-1$}: $\frac{b'_1 - b_1}{a'_1-a_1}<-1$ \text{and} $\frac{b'_2 - b_2}{a'_2-a_2}<-1$.

\item \emph{Both reaction vectors with positive slopes}: $\frac{b'_1 - b_1}{a'_1-a_1}>0$ \text{and} $\frac{b'_2 - b_2}{a'_2-a_2}>0$.
\item \emph{One reaction vector with positive slope and the other with slope between $-1$ and $0$}: $\frac{b'_2 - b_2}{a'_2-a_2}> 0$ \text{and} $-1<\frac{b'_1 - b_1}{a'_1-a_1}<0$.
\item \emph{One reaction vector with positive slope and the other with slope less than $-1$}: $\frac{b'_2 - b_2}{a'_2-a_2}>0$ \text{and} $\frac{b'_1 - b_1}{a'_1-a_1}<-1$.
\end{enumerate}

\begin{figure}[h!]

\begin{subfigure}[b]{.49\linewidth}
\includegraphics[width=\linewidth]{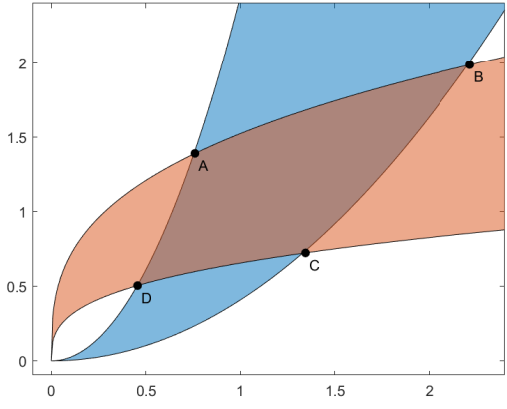}
\caption{Case (i).}\label{fig:case(i)}
\end{subfigure}
\begin{subfigure}[b]{.49\linewidth}
\includegraphics[width=\linewidth]{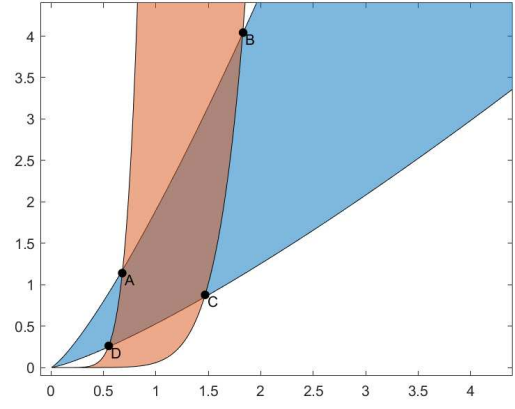}
\caption{Case (ii).}\label{fig:case(ii)}
\end{subfigure}

\begin{subfigure}[b]{.49\linewidth}
\includegraphics[width=\linewidth]{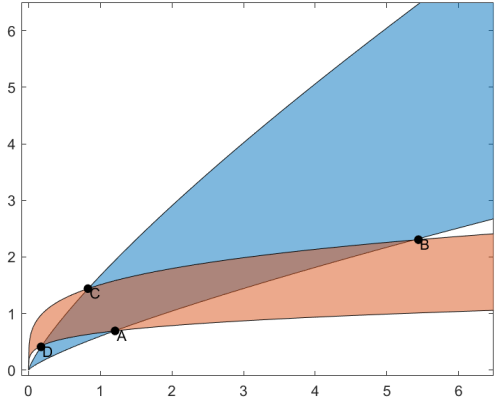}
\caption{Case (iii).}\label{fig:case(iii)}
\end{subfigure}
\begin{subfigure}[b]{.49\linewidth}
\includegraphics[width=\linewidth]{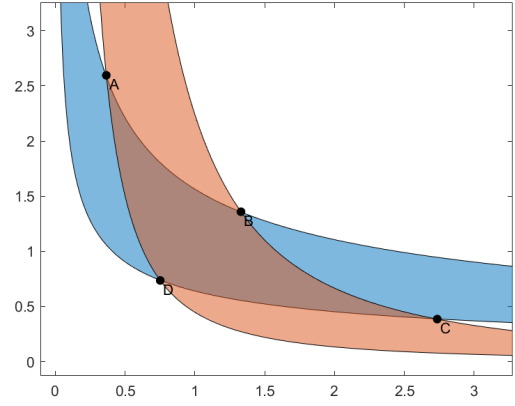}
\caption{Case (iv).}\label{fig:case(iv)}
\end{subfigure}

\caption{Uncertainty regions corresponding to two reversible reactions in case (i)-(iv).}\label{fig:four_cases}
\end{figure}

Figure~\ref{fig:four_cases} illustrates the uncertainty regions corresponding to cases (i)-(iv). Throughout this paper, the analysis for cases (i)-(iv) will be similar, while, for simplicity, in cases (v) and (vi) we will make the {\em additional assumption} that $\epsilon$ is small enough. In subsection~\ref{subsec:partial} we will discuss how to construct the minimal invariant region and minimal globally attracting region for cases (i)-(iv). For simplicity, we will only consider case (i), i.e. both reaction vectors with negative slopes; one with slope less than $-1$ and the other with slope greater than $-1$: $-1<\frac{b'_1 - b_1}{a'_1-a_1}<0$ \text{and} $\frac{b'_2 - b_2}{a'_2-a_2}< -1$; the analysis for other cases will proceed analogously. Throughout the next subsection, all references to the dynamical system will be with respect to case (i).

\subsection{Cases (i)-(iv)}\label{subsec:partial}

In the following lemma, we show that fixing the rate constants to certain values also fixes the omega-limit points of the trajectories corresponding to the variable-$k$ dynamical system generated by two reversible reactions.

\begin{figure}[h!]
\centering
\includegraphics[scale=0.6]{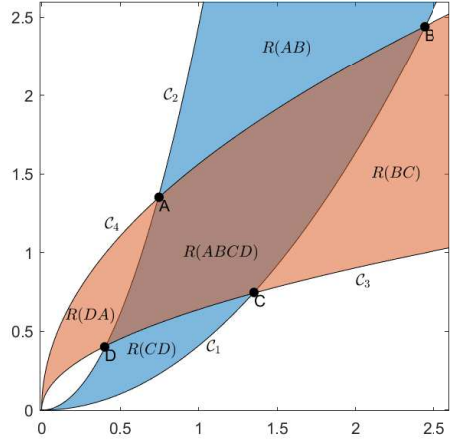}
\caption{Uncertainty regions corresponding to two reversible reactions. Regions $R_{AB},R_{BC},R_{CD},R_{DA},R_{ABCD}$ are also marked on the figure.}
\label{fig:invariant_region}
\end{figure}

\begin{lemma}\label{lem:global_attractors}
Consider the variable-$k$ dynamical system generated by Equation~(\ref{eq:variable-k}). Consider the following curves (see Fig.~\ref{fig:invariant_region} for an example):
\begin{enumerate}
\item[] $\mathcal{C}_1: x^{a'_1-a_1}y^{b'_1-b_1} = \frac{1}{\epsilon^2}$. 
\item[] $\mathcal{C}_2: x^{a'_1-a_1}y^{b'_1-b_1} = \epsilon^2$.
\item[] $\mathcal{C}_3: x^{a'_2-a_2}y^{b'_2-b_2} = \epsilon^2$. 
\item[] $\mathcal{C}_4: x^{a'_2-a_2}y^{b'_2-b_2} = \frac{1}{\epsilon^2}$.
\end{enumerate}
Let $A$ be the intersection point of the curves $\mathcal{C}_2$ and $\mathcal{C}_4$, $B$ be the intersection point of the curves $\mathcal{C}_1$ and $\mathcal{C}_4$, $C$ be the intersection point of the curves $\mathcal{C}_1$ and $\mathcal{C}_3$ and $D$ be the intersection point of the curves $\mathcal{C}_2$ and $\mathcal{C}_3$. Let $\vx(t)$ be a trajectory of~(\ref{eq:variable-k}) and let $\vx(0)\in\mathbb{R}^2_{>0}$. Then we have the following:
\begin{enumerate}[label=(\roman*)]
\item If $k_1(t)= \epsilon,k_2(t)=\frac{1}{\epsilon},k_3(t)=\frac{1}{\epsilon},k_4(t)=\epsilon$, then $\displaystyle\lim_{t\to\infty}\vx(t) = A$.
\item If $k_1(t)=\frac{1}{\epsilon},k_2(t)= \epsilon,k_3(t)=\frac{1}{\epsilon},k_4(t)=\epsilon$, then $\displaystyle\lim_{t\to\infty}\vx(t) = B$.
\item If $k_1(t)= \frac{1}{\epsilon},k_2(t)= \epsilon,k_3(t)=\epsilon,k_4(t)=\frac{1}{\epsilon}$, then $\displaystyle\lim_{t\to\infty}\vx(t) = C$.
\item If $k_1(t)= \epsilon,k_2(t)= \frac{1}{\epsilon},k_3(t)=\epsilon,k_4(t)=\frac{1}{\epsilon}$, then $\displaystyle\lim_{t\to\infty}\vx(t) = D$.
\end{enumerate}
\end{lemma}

\begin{proof}
Note that the point $A$ is the intersection of the curves $x^{a'_1-a_1}y^{b'_1-b_1} = \frac{1}{\epsilon^2}$ and $x^{a'_2-a_2}y^{b'_2-b_2} = \frac{1}{\epsilon^2}$. If $k_1(t)=\frac{1}{\epsilon},k_2(t)=\epsilon,k_3(t)=\frac{1}{\epsilon},k_4(t)= \epsilon$, then the point $A$ becomes detailed balanced since $k_1(t)x^{a_1}y^{b_1}=k_2(t)x^{a'_1}y^{b'_1}$ and $k_3(t)x^{a_2}y^{b_2}=k_4(t)x^{a'_2}y^{b'_2}$. Since the dynamical system is two-dimensional, it follows from~\cite{craciun2013persistence} that $\displaystyle\lim_{t\to\infty}\vx(t)= A$. 
\end{proof}

\textbf{Constructing} $\mathcal{M}_{\epsilon}$: Let us denote the intersection points of the curves $\mathcal{C}_1,\mathcal{C}_2,\mathcal{C}_3,\mathcal{C}_4$ by $A,B,C,D$ as in Figure~\ref{fig:invariant_region}. Note that two out of these four intersection points will have the cone (formed by the attracting directions at those point) that contains the region $R_{ABCD}$. Let us denote these points by $A$ and $C$. For the other two points, the cone (formed by the attracting directions at those points) is contained in the region $R_{ABCD}$. Let us denote these points by $B$ and $D$. Starting from points $A$ and $C$, choose rate constants so that the neighbouring intersection points are global attractors for these trajectories. The region enclosed by these four trajectories is $\mathcal{M}_{\epsilon}$.

\begin{proposition} \label{lem:contained_uncertainty}
Consider the dynamical system depicted in Figure~\ref{fig:invariant_region}. Then $\mathcal{M}_{\epsilon}\subset R_{AB}\cup R_{BC}\cup R_{CD}\cup R_{DA}\cup R_{ABCD}$.
\end{proposition}

\begin{proof}
Consider a trajectory $\vx(t)$ of the reaction network given by~(\ref{eq:variable-k}) with $\vx(0)=D$. The dynamical system it generates is given by
\[\begin{pmatrix}\label{eq:dynamical_system}
\dot{x}\\
\dot{y} 
\end{pmatrix} = \left(k_1(t)x^{a_1}y^{b_1} - k_2(t)x^{a'_1}y^{b'_1}\right)\begin{pmatrix}
a'_1 - a_1\\
b'_1 - b_1
\end{pmatrix}
+ 
\left(k_3(t)x^{a_2}y^{b_2} - k_4(t)x^{a'_2}y^{b'_2}\right)\begin{pmatrix}
a'_2 - a_2\\
b'_2 - b_2
\end{pmatrix}\]

If we choose rate constants $k_1(t)= \epsilon,k_2(t)= \frac{1}{\epsilon},k_3(t)= \frac{1}{\epsilon} ,k_4(t)= \epsilon$, then by Lemma~\ref{lem:global_attractors} we get $\displaystyle\lim_{t\to\infty}\vx(t) = A$. We now show that this trajectory stays inside the region $R_{DA}$. Since we are in case (i), we have $-1$: $-1<\frac{b'_1 - b_1}{a'_1-a_1}<0$ \text{and} $\frac{b'_2 - b_2}{a'_2-a_2}< -1$. Therefore, within the region $R_{DA}$, the trajectory is confined to a cone formed by $\vv_1= \begin{pmatrix} a'_1 - a_1\\
b'_1 - b_1
\end{pmatrix}$ and $\vv_2=\begin{pmatrix} a'_2 - a_2\\
b'_2 - b_2
\end{pmatrix}$ as shown in Figure~\ref{fig:cone}. Therefore, this trajectory cannot cross the curve $OD$. We show that it also cannot cross the curves $OA$ and $DA$. For contradiction, assume that the trajectory intersects $OA$ at point $P$. Note that since the point $P$ lies on the curve $\mathcal{C}_4$, we have $k_3(t)x^{a_2}y^{b_2} = k_4(t)x^{a'_2}y^{b'_2}$. Therefore, the only contribution to the vector field at point $P$ is due to the attracting direction of the blue uncertainty region (shown as $\vv_1$ in Figure~\ref{fig:cone}) which points towards the region $OAD$. A similar argument shows that the trajectory cannot cross the curve $DA$. Repeating this for other parts of the boundary of $\mathcal{M}_{\epsilon}$, we get that $\mathcal{M}_{\epsilon}\in R_{AB}\cup R_{BC}\cup R_{CD}\cup R_{DA}$, as required.

\end{proof}


\begin{figure}[h!]
\centering
\includegraphics[scale=0.6]{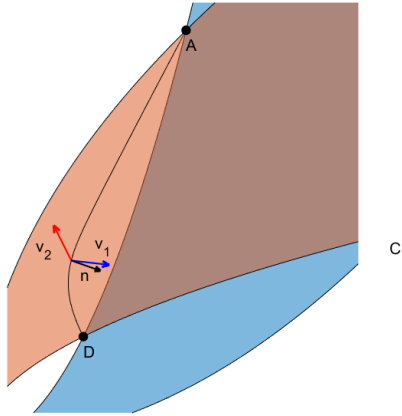}
\caption{The direction of the trajectory from $D$ to $A$ is confined to the cone formed by vectors $v_1$ and $v_2$. In particular, this means that this trajectory lies in the region $R_{DA}$ (shown in light orange color).}
\label{fig:cone}
\end{figure}

\begin{proposition}\label{prop:invariant}
$\mathcal{M}_{\epsilon}$ is an invariant region for the dynamical system generated by the reaction network in Equation~(\ref{eq:variable-k}).
\end{proposition}

\begin{proof} 

To show that $\mathcal{M}_{\epsilon}$ is an invariant region, it suffices to show that on the boundary of $\mathcal{M}_{\epsilon}$, the vector field points towards the interior of $\mathcal{M}_{\epsilon}$~\cite{nagumo1942lage,blanchini1999set}. Towards this, consider the dynamical system generated by Equation~(\ref{eq:variable-k})

\begin{eqnarray}\label{eq:invariant_region}
\begin{pmatrix}
\dot{x}\\
\dot{y} 
\end{pmatrix} = \left(k_1(t)x^{a_1}y^{b_1} - k_2(t)x^{a'_1}y^{b'_1}\right)\begin{pmatrix}
a'_1 - a_1\\
b'_1 - b_1
\end{pmatrix}
+ 
\left(k_3(t)x^{a_2}y^{b_2} - k_4(t)x^{a'_2}y^{b'_2}\right)\begin{pmatrix}
a'_2 - a_2\\
b'_2 - b_2
\end{pmatrix}
\end{eqnarray}
Let
\begin{eqnarray}
\vv_1=\begin{pmatrix} a'_1 - a_1 \\  b'_1 - b_1 \end{pmatrix}\, \rm{and}\,\, \vv_2=\begin{pmatrix} a'_2 - a_2 \\ b'_2 - b_2\end{pmatrix}. \end{eqnarray}
Then Equation~(\ref{eq:invariant_region}) can be written as

\begin{eqnarray}
\begin{pmatrix}
\dot{x}\\
\dot{y} 
\end{pmatrix} = \left(k_1(t)x^{a_1}y^{b_1} - k_2(t)x^{a'_1}y^{b'_1}\right)\vv_1
+ 
\left(k_3(t)x^{a_2}y^{b_2} - k_4(t)x^{a'_2}y^{b'_2}\right)\vv_2
\end{eqnarray}

We will show that on the boundary of $\mathcal{M}_{\epsilon}$ consisting of the trajectory from $D$ to $A$, the vector field points towards the interior of $\mathcal{M}_{\epsilon}$. The proof for other parts of the boundary of $\mathcal{M}_{\epsilon}$ will follow analogously. From Lemma~\ref{lem:global_attractors}, the trajectory from $D$ to $A$ is given by the following system of ODEs.

 \begin{eqnarray}\label{eq:extreme_trajectory}
    \begin{pmatrix}
          \dot{x} \\
          \dot{y} 
         \end{pmatrix}_{D\rightarrow A} =
          \bigg(\epsilon x^{a_1}y^{b_1}  - \frac{1}{\epsilon} x^{a'_1}y^{b'_1} \bigg)
         \vv_1     
          + \bigg(\frac{1}{\epsilon} x^{a_2}y^{b_2} - \epsilon x^{a'_2}y^{b'_2} \bigg)
         \vv_2
\end{eqnarray}
(where we have used the fact that $k_1(t) = \epsilon, k_2(t) = \frac{1}{\epsilon}, k_3(t) = \frac{1}{\epsilon}, k_4(t) = \epsilon$). Let $\textbf{n}$ denote the inward pointing normal to the trajectory given by  Equation~(\ref{eq:extreme_trajectory}). We will show that 
\begin{eqnarray}
\begin{pmatrix}
\dot{x} \\
\dot{y} 
\end{pmatrix} \cdot\textbf{n} \geq 0
\end{eqnarray}
Note that
\begin{eqnarray}\label{eq:normal}
\begin{pmatrix}
          \dot{x} \\
          \dot{y} 
         \end{pmatrix}_{D\rightarrow A}\cdot\textbf{n} = \bigg[\bigg( \epsilon x^{a_1}y^{b_1} - \frac{1}{\epsilon} x^{a'_1}y^{b'_1} \bigg)
         \vv_1     
          + \bigg(\frac{1}{\epsilon} x^{a_2}y^{b_2} - \epsilon x^{a'_2}y^{b'_2} \bigg)
         \vv_2\bigg]\cdot\textbf{n} = 0
\end{eqnarray}

Since within region $OAD$, the trajectory of the dynamical system is confined to the cone formed by vectors $\vv_1$ and $\vv_2$ (as shown in Figure~\ref{fig:cone}), we get $ \vv_1 \cdot\textbf{n}  > 0\, \rm{and}\, \vv_2 \cdot\textbf{n} < 0$. Noting that $\epsilon \leq k_1(t),k_2(t),k_3(t),k_4(t)\leq\frac{1}{\epsilon}$, we have
\begin{eqnarray}\label{eq:1}
\bigg(k_1(t)x^{a'_1}y^{b'_1} - k_2(t)x^{a_1}y^{b_1}   \bigg)
         \vv_1\cdot\textbf{n} \geq 
\bigg(\epsilon x^{a_1}y^{b_1} - \frac{1}{\epsilon}x^{a'_1}y^{b'_1}  \bigg)
         \vv_1\cdot\textbf{n}
\end{eqnarray}

and 

\begin{eqnarray}\label{eq:2}
\bigg(k_3(t)x^{a_2}y^{b_2} - k_4(t) x^{a'_2}y^{b'_2}   \bigg)
         \vv_2\cdot\textbf{n} \geq 
\bigg(\frac{1}{\epsilon}x^{a_2}y^{b_2} - \epsilon x^{a'_2}y^{b'_2}  \bigg)
         \vv_2\cdot\textbf{n} 
\end{eqnarray}
Adding Equations~(\ref{eq:1}) and~(\ref{eq:2}) and using Equation~(\ref{eq:normal}), we get that 

\begin{eqnarray}\label{eq:3}
\bigg[\bigg(k_1(t)x^{a_1}y^{b_1} - k_2(t)x^{a'_1}y^{b'_1}   \bigg)
         \vv_1 +
         \bigg(k_3(t)x^{a_2}y^{b_2} - k_4(t) x^{a'_2}y^{b'_2}   \bigg)
         \vv_2\bigg]
         \cdot\textbf{n}\geq 0
\end{eqnarray}
as required. 
\end{proof}

\begin{remark}
Note that Proposition~\ref{prop:invariant} shows that $\mathcal{M}_{\epsilon_0}$ is an invariant region for $\epsilon_0>0$. For all $\epsilon < \epsilon_0$, the inequalities given by Equations~(\ref{eq:1}),~(\ref{eq:2}) and~(\ref{eq:3}) become strict and hence the net vector field along the boundary of $\mathcal{M}_{\epsilon}$ points strictly onto its interior. 
\end{remark}

\begin{remark}\label{rem:continuous_dependence}
Consider points $P_1,P_2,P_3\in\mathbb{R}^2_{>0}$. If we have $P_1\leadsto P_2$ and $P_2\leadsto P_3$, then we have $P_1\leadsto P_3$ since the solutions of this dynamical system depend continuously on their initial conditions.
\end{remark}

\begin{proposition}\label{prop:path_invariant}
Consider the dynamical system generated by Equation~(\ref{eq:variable-k}). If $P_2\in\mathcal{M}_{\epsilon}$, then $P_1\leadsto P_2$ for any $P_1\in\mathbb{R}^2_{>0}$.
\end{proposition}

\begin{proof}

We proceed by case analysis. (Refer to Figure~\ref{fig:invariant_region}).

\begin{enumerate}[label=(\roman*)]

\item $P_2\in R_{ABCD}$: Then $P_2$ is the intersection of the curves $k_2x^{a_1}y^{b_1} = k_1x^{a'_1}y^{b'_1}$ and $k_4x^{a_2}y^{b_2} = k_3x^{a'_2}y^{b'_2}$ for some constants $\epsilon\leq k_1,k_2,k_3,k_4\leq\frac{1}{\epsilon}$. Choosing $k_1(t)=k_1,k_2(t)=k_2,k_3(t)=k_3,k_4(t)=k_4$, we get that $P$ is detailed balanced for these choice of rate constants. Noting that the dynamical system is two-dimensional, it follows from~\cite{craciun2013persistence} that $P_1\leadsto P_2$.

\item $P_2\in \mathcal{M}_{\epsilon}\setminus R_{ABCD}$: Without loss of generality, assume that the point $P_2$ lies in the region $OAD$ (Similar arguments will work in the other regions). Consider a trajectory $\vx(t)$ of this dynamical system with $\vx(0)=P_1$ and choose rate constants $k_1(t)= \epsilon, k_2(t)= \frac{1}{\epsilon},k_3(t)= \epsilon, k_4(t)= \frac{1}{\epsilon}$ as in Lemma~\ref{lem:global_attractors} so that $\displaystyle\lim_{t\to\infty}\vx(t) = D$. This implies that $P_1\leadsto D$. Now starting close to $D$, choose rate constants $k_1(t)= \epsilon,k_2(t)= \frac{1}{\epsilon},k_3(t)=\frac{1}{\epsilon} ,k_4(t)=\epsilon$, so that $\displaystyle\lim_{t\to\infty}\vx(t) = A$. Construct a line in the attracting direction of the blue uncertainty region that passes through $P_2$. Let this line intersect the curve $DA$ at point $Q$. Now starting close to $D$, choose rate constants $k_1(t)= \epsilon, k_2(t)= \frac{1}{\epsilon},k_3(t)=\frac{1}{\epsilon},k_4(t)=\epsilon$ and follow the trajectory till it reaches the point $Q$. Now set the rate constants of the reaction corresponding to the red uncertainty region such that $k_3(t)x^{a_2}y^{b_2} = k_4(t)x^{a'_2}y^{b'_2} $. This means the only vector field at point $Q$ is due to the attracting direction of the blue uncertainty region. Trace this trajectory till we get to the point $P_2$. From Remark~\ref{rem:continuous_dependence}, we get that $P_1\leadsto P_2$.
\end{enumerate}

\end{proof}

\begin{theorem}\label{thm:min_invariant_region}
$\mathcal{M}_{\epsilon}$ is the minimal invariant region for the dynamical system generated by the reaction network in Equation~(\ref{eq:variable-k}).
\end{theorem}

\begin{proof}
Note that Proposition~\ref{prop:invariant} shows that $\mathcal{M}_{\epsilon}$ is an invariant region for the dynamical system generated by the reaction network in Equation~(\ref{eq:variable-k}). To show that $\mathcal{M}_{\epsilon}$ is the invariant region, we prove that $\mathcal{M}_{\epsilon}$ is contained in every invariant region of the dynamical system. This follows from 
Proposition~\ref{prop:path_invariant}. 

\end{proof}

\subsection{Cases (v)-(vi)}

The goal of this section is to construct the region $\mathcal{M}_{\epsilon}$ for cases (v) and (vi), where one reaction vector has positive slope and the other has negative slope. Figure~\ref{fig:four_cases} illustrates the uncertainty regions corresponding to cases (i)-(iv).

\begin{figure}[h!]

\begin{subfigure}[b]{.49\linewidth}
\includegraphics[width=\linewidth]{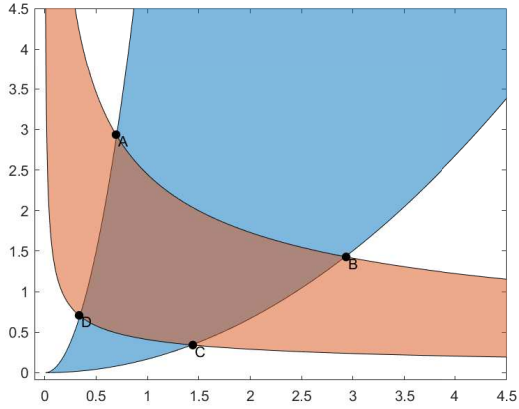}
\caption{Case (v).}\label{fig:case(v)}
\end{subfigure}
\begin{subfigure}[b]{.49\linewidth}
\includegraphics[width=\linewidth]{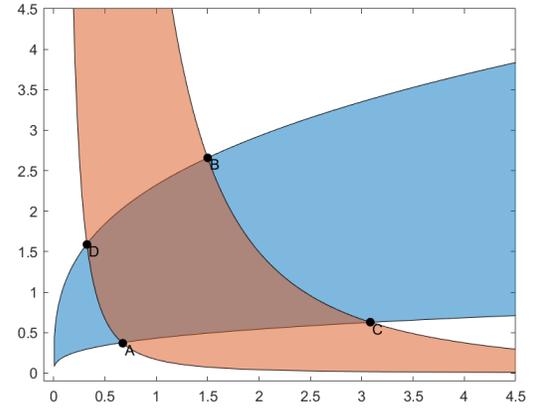}
\caption{Case (vi).}\label{fig:case(vi)}
\end{subfigure}

\caption{Uncertainty regions corresponding to two reversible reactions in case (v)-(vi).}\label{fig:last_two_cases}
\end{figure}

In what follows, we present the analysis of case (v); the analysis for case (vi) is completely analogous. \\

Consider the reaction network given in Equation~(\ref{eq:variable-k}). The dynamical system it generates is given by
\[\begin{pmatrix}\label{eq:dynamical_system_new}
\dot{x}\\
\dot{y} 
\end{pmatrix} = \left(k_1(t)x^{a_1}y^{b_1} - k_2(t)x^{a'_1}y^{b'_1}\right)\begin{pmatrix}
a'_1 - a_1\\
b'_1 - b_1
\end{pmatrix}
+ 
\left(k_3(t)x^{a_2}y^{b_2} - k_4(t)x^{a'_2}y^{b'_2}\right)\begin{pmatrix}
a'_2 - a_2\\
b'_2 - b_2
\end{pmatrix}\]

For convenience, we will denote $p_1=a_1 - a'_1, q_1=b'_1 - b_1, p_2=a_2 - a'_2, q_2=b'_2 - b_2$. Without loss of generality, assume that $p_1=a_1 - a'_1>q_1=b'_1 - b_1>0$, $p_2=a_2 - a'_2<0$, $q_2=b'_2 - b_2>0$. Consider the following intersection points as shown in Figure~\ref{fig:last_two_cases}.

\begin{enumerate}[label=(\roman*)]
\item $A: y^{q_1} =\frac{1}{\epsilon^2}x^{p_1}\, \text{and}\, y^{q_2} =\frac{1}{\epsilon^2}x^{p_2}$. 
\item $B: y^{q_1} = \epsilon^2x^{p_1}\, \text{and}\, y^{q_2} =\frac{1}{\epsilon^2}x^{p_2}$. 
\item $C: y^{q_1} = \epsilon^2x^{p_1}\, \text{and}\, y^{q_2} = \epsilon^2x^{p_2}$. 
\item $D: y^{q_1} =\frac{1}{\epsilon^2}x^{p_1}\, \text{and}\, y^{q_2} = \epsilon^2x^{p_2}$. 
\end{enumerate}

Table~\ref{table2} shows the slopes of the tangents to the boundary of the uncertainty regions at their intersection points. We split our analysis into three subcases depending on the sign of $p_1+p_2-q_1-q_2$.

\begin{table}[htbp]
	\centering 
	\caption{Slope of the tangents at the intersection of uncertainty regions in the limit $\epsilon\to 0$}  
	\label{table2}  
	\begin{tabular}{|c|c|}  
		\hline  
	 
		$m_{A}^{p_1,q_1}$ & $\frac{p_1}{q_1}\epsilon^{\frac{2(-p_1+p_2 + q_1 -q_2)}{p_1q_2 - p_2q_1}}$ \\
		
        \hline
        
        $m_{A}^{p_2,q_2}$ & $\frac{p_2}{q_2}\epsilon^{\frac{2(-p_1+p_2 + q_1 -q_2)}{p_1q_2 - p_2q_1}}$  \\
        
        \hline
        
        $m_{B}^{p_1,q_1}$ & $\frac{p_1}{q_1}\epsilon^{\frac{2(-p_1 - p_2 + q_1 + q_2)}{p_1q_2 - p_2q_1}}$ \\

        \hline

        $m_{B}^{p_2,q_2}$ &  $\frac{p_2}{q_2}\epsilon^{\frac{2(-p_1 - p_2 + q_1 + q_2)}{p_1q_2 - p_2q_1}}$ \\
        
        \hline
        
        $m_{C}^{p_1,q_1}$ & $\frac{p_1}{q_1}\epsilon^{\frac{2(p_1 - p_2 - q_1 + q_2)}{p_1q_2 - p_2q_1}}$ \\
        
        \hline

		$m_{C}^{p_2,q_2}$ & $\frac{p_2}{q_2}\epsilon^{\frac{2(p_1 - p_2 - q_1 + q_2)}{p_1q_2 - p_2q_1}}$ \\
		
		\hline
		
        $m_{D}^{p_1,q_1}$ &  $\frac{p_1}{q_1}\epsilon^{\frac{2(p_1 + p_2 - q_1 - q_2)}{p_1q_2 - p_2q_1}}$ \\
        
        \hline
        
        $m_{D}^{p_2,q_2}$  & $\frac{p_2}{q_2}\epsilon^{\frac{2(p_1 + p_2 - q_1 - q_2)}{p_1q_2 - p_2q_1}}$ \\

		\hline
		
\end{tabular}
\end{table}

\begin{itemize}

\item[Case (a):] $p_1+p_2-q_1-q_2<0$. 

In this case, note that $m^{p_2,q_2}_{D}\to -\infty$ and $m^{p_2,q_2}_{C}\to 0$ as $\epsilon\to 0$. Since the slope of the tangent to the lower red curve varies continuously as we traverse from along the curve $D$ to $C$, there exists a point $E$ at which the slope of the tangent to the red curve has the same slope as the attracting direction of the blue uncertainty region. We construct trajectories of this dynamical system starting from the point $E$ that go towards $C$ and $D$. We now claim that both these trajectories stay inside the region $R(CD)$. Figure~\ref{fig:trajectory_C_D} illustrates this point.

We show that the trajectory cannot cross the curve $OC$. For contradiction, assume that the trajectory intersects $OC$ at point $P$. Note that since the point $P$ lies on the boundary of the blue uncertainty region, the vector field at point $P$ is given by the red attracting direction which points towards the interior of the region $R_{CD}$. We now show that the trajectory cannot cross the curves $EC$ and $ED$. For contradiction, assume that the trajectory intersects $EC$ at point $P'$. Note that the slope of the tangents to the lower red curve increase monotonically from $E$ to $C$. Therefore, the net vector field at point $P'$ which is in the blue attracting direction points towards the interior of the region $R_{CD}$. A similar argument can be used to show that the trajectory cannot intersect the curve $ED$.

\begin{figure}[h!]
\centering
\includegraphics[scale=0.6]{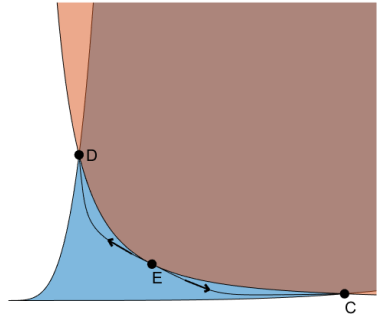}
\caption{In subcase (a), there exists a point $E$ on the curve from $D$ to $C$ where the slope of the tangent to the red curve has the same slope as the attracting direction of the blue uncertainty region. The boundary of $\mathcal{M}_{\epsilon}$ in this region is given by trajectories that start from $E$ and go towards $C$ and $D$.}
\label{fig:trajectory_C_D}
\end{figure}

We now show that the trajectory cannot cross the curve $OD$. One can calculate the coordinate of $E$ to be the following: $x_E=\left(-\frac{q_1 q_2}{p_1 p_2}\right)^{\frac{q_2}{p_2-q_2}} \epsilon^{-\frac{2}{p_2-q_2}},y_E=\left(-\frac{q_1q_2}{p_1p_2}\right)^{\frac{p_2}{p_2-q_2}}\epsilon^{-\frac{2}{p_2-q_2}}$. In the discussion that follows, please refer Figure~\ref{fig:not_going_outside}. Extend the tangent at $E$ so that it meets $OD$ at point $H$. Let $F$ be the point on the curve $OD$, where slope of the tangent is equal to the slope of the attracting direction corresponding to the red uncertainty region. The coordinate of $F$ is given by the following: $x_F=\bigg(-\frac{q_1 q_2}{p_1 p_2}\bigg)^{\frac{q_1}{p_1-q_1}} \epsilon^{\frac{2}{p_1-q_1}},y_F=\bigg(-\frac{q_1 q_2}{p_1 p_2}\bigg)^{\frac{p_1}{p_1-q_1}} \epsilon^{\frac{2}{p_1-q_1}-\frac{2}{p_1}+\frac{2}{q_1}}$. We now show that $y_F$ is lesser than the y-coordinate of the point $H$. Since $p_1+p_2-q_1-q_2<0$ and $p_1 >q_1$, we have $-\frac{2}{p_2-q_2}<\frac{2}{p_1-q_1}<\frac{2}{p_1-q_1}-\frac{2}{p_1}+\frac{2}{q_1}$. Therefore, we get $y_F-(-\frac{q_1}{p_1}) x_F-[y_E-(-\frac{q_1}{p_1}) x_E] =\frac{1}{p_1}[p_1 y_F + q_1 x_F - p_1 y_E - q_1 x_E]\to -\frac{1}{p_1}(p_1 y_E + q_1 x_E) \quad \text{as } \epsilon\to 0$. Note that $-\frac{1}{p_1}(p_1 y_E + q_1 x_E)<0$, therefore we get $y_F + \frac{q_1}{p_1} x_F < y_E+ \frac{q_1}{p_1} x_E$ as $\epsilon\to 0$.

\begin{figure}[h!]
\centering
\includegraphics[scale=0.6]{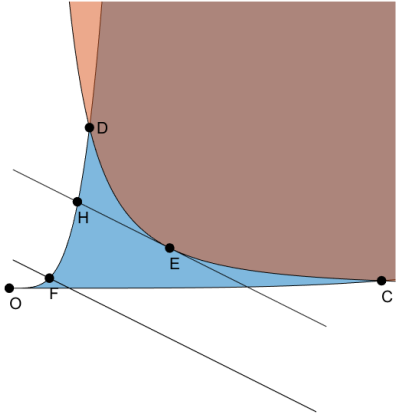}
\caption{In the same setting as in Fig.~\ref{fig:trajectory_C_D}, we now focus on relative positions of some important attracting directions lines. The point $F$ is chosen such that the slope of the tangent line to the curve $OD$ at $F$ is the same as the slope of the attraction direction of the red uncertainty region.}
\label{fig:not_going_outside}
\end{figure}

Given the cone formed by the attracting directions in region $EHD$, the trajectory always remains in the region $EHD$. Note that the slope of the tangents to the upper blue curve increases monotonically from $O$ to $D$, the red attracting direction will point towards the interior of the region $R_{CD}$ from $F$ to $D$. Suppose that the trajectory meets the curve $OD$ at the point $P'$. Since $y_F$ is lesser than the y-coordinate of the point $H$, we get that the red attracting direction will point towards the interior of the region $R_{CD}$ from $H$ to $D$. In particular, at $P'$, the vector field points towards the interior of the region $R_{CD}$. As a consequence, the trajectory cannot cross $OD$. 

\begin{figure}[h!]
\centering
\includegraphics[scale=0.6]{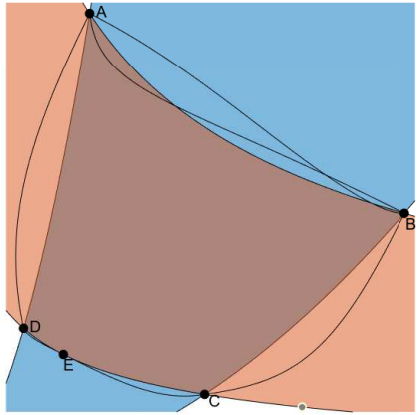}
\caption{Boundary of $\mathcal{M}_{\epsilon}$ for subcase(a) in cases (v)-(vi).}
\end{figure} 
    
We now prove that there exists trajectories from $A$ to $D$ and from $B$ to $C$ which stay inside the regions $R_{AD}$ and $R_{BC}$ respectively. Note that the slope of the tangents to the upper blue curve at $A$ and $D$ given by $m^{p_1,q_1}_{A}$ and $m^{p_1,q_1}_{D}$ satisfy $m^{p_1,q_1}_{A}\to \infty$ and $m^{p_1,q_1}_{D}\to \infty$ as $\epsilon\to 0$. Further, the slope of the tangents to the upper blue curve increase monotonically from $D$ to $A$. For contradiction, assume that the trajectory from $A$ to $D$ intersects the curve $AD$ at point $P$. Then the vector field at point $P$ is given by the attracting direction corresponding to the lower red curve, which points towards the interior of the region $R_{AD}$.

Note that inside the region $R_{AD}$, the trajectory from $A$ to $D$ is confined to the relevant cone formed by the attracting directions of the uncertainty regions. Therefore the trajectory cannot intersect the upper red curve from $A$ to above it. We now show that the trajectory also cannot intersect the lower red curve from $D$ to upwards. 
Further, the slopes of the tangents to the lower red curve decrease monotonically from $D$ to upwards. For contradiction, assume that the trajectory from $A$ to $D$ intersects the lower red curve from $D$ to upwards at point $P'$. The vector field at $P'$ is given by the blue attracting direction, which points towards the region $R_{AD}$. A similar argument can be made to show that the trajectory from $B$ to $C$ stays inside the region $R_{BC}$.

We now show how to construct the boundary of $\mathcal{M}_{\epsilon}$ in the region $R_{AB}$. Note that the slopes of the tangents to upper red curve satisfy $m^{p_2,q_2}_{A}\to-\infty$ and $m^{p_2,q_2}_{B}\to 0$ as $\epsilon\to 0$. Consider the two trajectory that starts at $A$ and ends at $B$ and the trajectory that starts at $B$ and ends at $A$. We will consider the outer union of these trajectories. We show that the intersection of these two trajectories cannot lie in the region $ABCD$. Note that in the limit $\epsilon\to 0$, on the curve $AB$, the slope of the tangent changes continuously on the interval $(-\infty,0)$. The blue attracting direction has a fixed negative slope given by $-\frac{p_1}{q_1}$. Therefore, both trajectories from the point $A$ to the point $B$ and from the point $B$ to the point $A$ will enter the blue uncertainty region. Let us assume that they intersect at point $M$. We will show that the trajectories $AM$ and $MB$ will form a part of the boundary of $\mathcal{M}_{\epsilon}$. To show this, we will prove that the point $M$ lies outside the region $ABCD$. The trajectories from $A$ to $B$ and from $B$ to $A$ are initially both outside the region $ABCD$. To enter the region $ABCD$, we need the slope of the tangent to the upper red curve to be greater than $-\frac{p_1}{q_1}$ for the trajectory from $A$ to $B$, and to be $<-\frac{p_1}{q_1}$ for the trajectory from $B$ to $A$. Since this cannot be achieved simultaneously, the intersection of the trajectories $AB$ and $BA$ cannot be inside the region $ABCD$.

\item[(b)] $p_1+p_2-q_1-q_2>0$. 
      
In this case, we have from Table~\ref{table2}, $m^{p_1,q_1}_{A}\to \infty, m^{p_2,q_2}_{A}\to -\infty, m^{p_1,q_1}_{B}\to \infty, m^{p_2,q_2}_{B}\to -\infty, m^{p_1,q_1}_{C}\to 0, m^{p_2,q_2}_{C}\to 0, m^{p_1,q_1}_{D}\to 0, m^{p_2,q_2}_{D}\to 0$ as $\epsilon\to 0$. This is analogous to the subcase (a) and we can show that for $\epsilon$ small enough, the boundary of $\mathcal{M}_{\epsilon}$ is given by the following trajectories:
\begin{enumerate}
\item From $A$ to $B$.
\item There exists a point $E$ on the curve $BC$ such that the slope of the tangent to the blue curve has the same slope as the attracting direction of the red uncertainty region. Now construct trajectories from $E$ to $B$ and $E$ to $C$.
\item From $C$ to $D$.
\item Outer union of the trajectories from $A$ to $D$ and $D$ to $A$.
\end{enumerate}

\item[(c)] $p_1+p_2-q_1-q_2=0$, this will be a combination of the previous situations.

From Table~\ref{table2}, we have the following: $m^{p_1,q_1}_B = m^{p_1,q_1}_D = \frac{p_1}{q_1}$ and $m^{p_2,q_2}_B = m^{p_2,q_2}_D = \frac{p_2}{q_2}$.  Let $E$ be a point on the lower red curve $CD$ such that the tangent at $E$ has the same slope as the blue attracting direction, and $E'$ be a point on the upper red curve $AB$ such that the tangent at $E'$ has the same slope as the blue attracting direction. Let us assume that the slope of the tangent on the lower red curve at point $C$ is $k$. Note that the slopes of the tangents decrease monotonically on the curve $CD$ in the range $[k,\frac{p_2}{q_2}]$. Similarly, on the upper curve $AB$, the slopes of the tangents decrease monotonically in the range $[\frac{p_2}{q_2},\frac{\lambda}{k}]$, where $\lambda$ is some positive constant. Therefore, there is at most one point $E$ or $E'$ on the red curves $AB$ or $CD$. When $\epsilon$ is small enough, we can construct the boundary of $\mathcal{M}_{\epsilon}$, where the upper trajectory between $A$ and $B$ is like case (a) while the lower trajectory between $C$ and $D$ is like case (b); or the upper trajectory between $A$ and $B$ is like case (b) while the lower trajectory between $C$ and $D$ is like case (a). Similarly on the blue curves $AD$ and $BC$, we have atmost one special point $F$ or $F'$, where the slope of the tangent is same as the slope of the red attracting direction. The construction of the boundary of $\mathcal{M}_{\epsilon}$ then proceeds in identical fashion as described above.

\end{itemize}

\begin{theorem}
Consider a dynamical system generated by Equation~(\ref{eq:variable-k}). Then, for $\epsilon$ small enough, $\mathcal{M}_{\epsilon}$ is the minimal invariant region.
\end{theorem}

\begin{proof}
The proof proceeds in identical fashion to Proposition~\ref{prop:invariant} and Theorem~\ref{thm:min_invariant_region}.
\end{proof}

In what follows next, we show that for $\epsilon$ small enough, $\mathcal{M}_{\epsilon}$ is also the minimal globally attracting region. Towards this, we need to analyze the points amongst $(A,B,C,D)$, that are end points of trajectories which form the boundary of $\mathcal{M}_{\epsilon}$. In particular, for every $\epsilon>0$, we are interested in the angle that the trajectories that form the boundary of $\mathcal{M}_{\epsilon}$ make when they meet at the globally attracting points. To make this analysis work, it is useful to linearize the dynamical system and study the eigenvalues of the corresponding Jacobian. The next proposition makes this precise.

The Jacobian corresponding to the dynamical system~(\ref{eq:variable-k}) is given by $J=\begin{pmatrix}
J_{11} & J_{12}\\
J_{21} & J_{22}
\end{pmatrix}$ 
where
\begin{eqnarray}\label{eq:jacobian}
\begin{split}
J_{11} &= (a'_1 - a_1)a_1k_1 x^{a_1-1}y^{b_1} - (a'_1 - a_1)a'_1k_2 x^{a'_1-1}y^{b'_1} + (a'_2 - a_2)a_2k_3 x^{a_2-1}y^{b_2} - (a'_2 - a_2)a'_2k_4 x^{a'_2-1}y^{b'_2}. \\
J_{12} &= (a'_1 - a_1)b_1k_1 x^{a_1}y^{b_1-1} - (a'_1 - a_1)b'_1k_2 x^{a'_1}y^{b'_1-1} + (a'_2 - a_2)b_2k_3 x^{a_2}y^{b_2-1} - (a'_2 - a_2)b'_2k_4 x^{a'_2}y^{b'_2-1}.\\
J_{21} &= (b'_1 - b_1)a_1k_1 x^{a_1-1}y^{b_1} - (b'_1 - b_1)a'_1k_2 x^{a'_1-1}y^{b'_1} + (b'_2 - b_2)a_2k_3 x^{a_2-1}y^{b_2} - (b'_2 - b_2)a'_2k_4 x^{a'_2-1}y^{b'_2}.\\
J_{22} &= (b'_1 - b_1)b_1k_1 x^{a_1}y^{b_1-1} - (b'_1 - b_1)b'_1k_2 x^{a'_1}y^{b'_1-1} + (b'_2 - b_2)b_2k_3 x^{a_2}y^{b_2-1} - (b'_2 - b_2)b'_2k_4 x^{a'_2}y^{b'_2-1}.\\
\end{split}
\end{eqnarray}

\begin{proposition}\label{prop:finite_equal_eigenvalues}
Consider case (i) and the trajectory from $D$ to $A$. (A similar analysis will apply to other cases). Let $J=\begin{pmatrix}
J_{11} & J_{12}\\
J_{21} & J_{22}
\end{pmatrix}$ be the Jacobian corresponding to the linearized dynamical system of this trajectory at point $A$. As $\epsilon$ varies, $J$ can only have equal eigenvalues at finitely many points.
\end{proposition}

\begin{proof}
Note that for $J$ to have equal eigenvalues, it has to satisfy
\begin{eqnarray}\label{eq:equal_eigenvalues}
(J_{11} + J_{22})^2 = 4(J_{11}J_{22} - J_{12} J_{21})
\end{eqnarray}
From Lemma~\ref{lem:global_attractors}, the point $A$ is the intersection of the following curves
\begin{eqnarray}\label{eq:intersection_A}
\begin{split}
  \frac{1}{\epsilon}  x^{a'_1}y^{b'_1}  & =  \epsilon x^{a_1}y^{b_1}  \\
  \epsilon  x^{a_2}y^{b_2}  & = \frac{1}{\epsilon} x^{a'_2}y^{b'_2}  
\end{split}
\end{eqnarray}
Solving Equations~(\ref{eq:equal_eigenvalues}) and~(\ref{eq:intersection_A}), we get a quasi-polynomial equation in $\epsilon$, which has finitely many roots. Therefore, the number of points where $J$ has equal eigenvalues are finite.

\end{proof}

The next proposition says that for trajectories that form the boundary of $\mathcal{M}_{\epsilon}$, certain directions are forbidden.

\begin{proposition}\label{prop:fast_eigendirection}
Consider case (i) and the trajectory from $D$ to $A$. (A similar analysis will apply to other cases). Let $J=\begin{pmatrix}
J_{11} & J_{12}\\
J_{21} & J_{22}
\end{pmatrix}$ be the Jacobian corresponding to the linearized dynamical system of this trajectory at point $A$. Then the trajectory approaches the point $A$ along the slower (smaller in magnitude) eigendirection of the Jacobian $J$.
\end{proposition}

\begin{proof}

Note that the the trajectory approaches the point $A$ along the slower(smaller in magnitude) eigendirection unless it lies on the faster (larger in magnitude) eigendirection. We show that the trajectory cannot approach the point $A$ along the faster eigendirection. In particular, we show that the faster eigendirection lies in the second or fourth quadrant centred at $A$ (refer to Figure~\ref{fig:cone}), which is forbidden by Proposition~\ref{lem:contained_uncertainty}.  

It is known~\cite[Theorem 14.3.4]{feinberg2019foundations}  that given a detailed balanced dynamical system with a positive steady state $c^*$, the Jacobian is symmetric with respect to the inner product given by $u* w = \frac{u_1 w_1}{c^*_1} + \frac{u_2 w_2}{c^*_2}$. There is a change of basis transformation that takes the Jacobian $J$ in the standard basis to the Jacobian $J^*$ that is symmetric with respect to this inner product, given by $J^* = P^{-1}J P$ where $P=\begin{pmatrix}
\sqrt{c^*_1} & 0 \\
0 & \sqrt{c^*_2}
\end{pmatrix}$
Note that signs of each element of $J$ is unchanged by this transformation. Using Lemma~\ref{lem:global_attractors} and the Jacobian given by Equation~(\ref{eq:jacobian}), the off-diagonal elements $J_{12}$ and $J_{21}$ of the Jacobian at $A$ are given by the following
\begin{itemize}
\item $J_{12} = (a'_1 - a_1)b_1\frac{1}{\epsilon} x^{a_1}y^{b_1-1} - (a'_1 - a_1)b'_1 \epsilon x^{a'_1}y^{b'_1-1} + (a'_2 - a_2)b_2 \frac{1}{\epsilon} x^{a_2}y^{b_2-1} - (a'_2 - a_2)b'_2 \epsilon x^{a'_2}y^{b'_2-1}.$
\item $J_{21} = (b'_1 - b_1)a_1\frac{1}{\epsilon} x^{a_1-1}y^{b_1} - (b'_1 - b_1)a'_1\epsilon x^{a'_1-1}y^{b'_1} + (b'_2 - b_2)a_2\frac{1}{\epsilon} x^{a_2-1}y^{b_2} - (b'_2 - b_2)a'_2\epsilon x^{a'_2-1}y^{b'_2}$
\end{itemize}
At point $A$, we have $\frac{1}{\epsilon} x^{a_1}y^{b_1} = \epsilon x^{a'_1}y^{b'_1}$ and $\frac{1}{\epsilon} x^{a_2}y^{b_2} = \epsilon x^{a'_2}y^{b'_2}$. Since we are in case (i), we have $(a'_1 - a_1)(b_1 - b'_1) >0$ and $(a'_2 - a_2)(b_2 - b'_2) >0$. Therefore, we get $J_{12}>0$ and $J_{21}>0$. This implies that $J^*_{12} = J^*_{21}>0$. The eigenvector corresponding to the smaller eigenvalue for a symmetric $2\times 2$ matrix is given by $e_1 =\begin{pmatrix}
\frac{J^*_{11} - J^*_{22} - \Delta}{2J^*_{12}} \\ 1
\end{pmatrix}$
where $\Delta = \sqrt{(J^*_{11} - J^*_{22})^2 + 4{J^*}^2_{12}}$. Since $\frac{J^*_{11} - J^*_{22} + \Delta}{2J^*_{12}}<0$, this eigenvector points either in the second or fourth quadrant. Transforming this eigenvector to the standard basis using $P^{-1}e_1$ does not change the sign of the elements of the eigenvector. Therefore, the vector corresponding to the faster eigendirection lies either in the second or fourth quadrant centred at $A$, and we are done.
\end{proof}

\begin{theorem}\label{thm:globally_attracting}
Consider a dynamical system generated by Equation~(\ref{eq:variable-k}). Then, for $\epsilon$ small enough, $\mathcal{M}_{\epsilon}$ is a globally attracting region.
\end{theorem}

\begin{proof}
Let $\epsilon_0$ be small enough so that the region $\mathcal{M}_{\epsilon_0}$ can be constructed according to the procedure described in Section 2. Note that $\mathcal{M}_{\epsilon}$ varies continuously as a function of $\epsilon$. In addition, we have  $\displaystyle\bigcup_{\epsilon\in(0,\epsilon_0]}\mathcal{M}_{\epsilon}=\mathbb{R}^2_{>0}$. Let $\zeta$ be small enough so that the region $\mathcal{M}_{\epsilon_0 + \zeta}$ can still be constructed. Let $\vx(t)$ be a solution of $\GG_{\epsilon}^{\rm{variable\mbox{-}k}}$ with $\vx(0)\in\mathbb{R}^2_{>0}$. Since $\displaystyle\bigcup_{\epsilon\in(0,\epsilon_0+\zeta]}\mathcal{M}_{\epsilon+\zeta}=\mathbb{R}^2_{>0}$, one can choose $\epsilon_1$ with $0<\epsilon_1<\epsilon_0$ such that $\vx(0)\in\displaystyle\bigcup_{\epsilon\in[\epsilon_1,\epsilon_0+\zeta]}\mathcal{M}_{\epsilon}$. We will prove that $\vx(t)\in \mathcal{M}_{\epsilon_0}$ for a large enough $t$.

Towards this, let $\partial\mathcal{M}_{\epsilon}$ denote the boundary of $\mathcal{M}_{\epsilon}$. Define a function $\Gamma:\displaystyle\bigcup_{\epsilon\in[\epsilon_1,\epsilon_0+\zeta]}\partial\mathcal{M}_{\epsilon}\rightarrow \bigg[\frac{1}{\epsilon_0+\zeta},\frac{1}{\epsilon_1}\bigg]$ so that $\Gamma(x,y)=\frac{1}{\epsilon}$ if $(x,y)\in\mathcal{M}_{\epsilon}$. We will show that $\Gamma(\vx(t))\leq\frac{1}{\epsilon_0}$ for a large enough $t$. Let us assume that this is not true. By Proposition~\ref{prop:invariant}, we know that the sets $\Gamma^{-1}\bigg(0,\frac{1}{\epsilon_0}\bigg]=\mathcal{M}_{\epsilon_0}$ and $\Gamma^{-1}\bigg(0,\frac{1}{\epsilon_1}\bigg]=\mathcal{M}_{\epsilon_1}$ are invariant. This implies that $\Gamma(\vx(t))\in[\frac{1}{\epsilon_0},\frac{1}{\epsilon_1}]$ for all $t\geq 0$.\\

Note that the function $\Gamma$ is differentiable everywhere except \emph{maybe} on boundary of $\mathcal{M}_{\epsilon}$, where trajectories end or where trajectories can start or intersect. We will handle these cases separately. We  will denote the curve that contains such points where $\Upsilon$ is not differentiable by $C_j(\epsilon)$. \\

Case I: Consider points on the boundary of $\mathcal{M}_{\epsilon}$ where trajectories can start or intersect. Note that in this case, the angle made by $\mathcal{M}_{\epsilon}$ along $C_j$ is always greater than $\pi$ no matter what $\epsilon$ is (this follows from analyzing cases (v) and (vi)). We will use some machinery from convex analysis. Towards this, for each curve $C_j$, let $\Upsilon_{j1}$ and $\Upsilon_{j2}$ be two functions such that $\Upsilon=\Upsilon_1$ on one side of $C_j$ and $\Upsilon=\Upsilon_2$ on the other side. We now consider the following cases. We let $-\Upsilon(\vx)=\max(-\Upsilon_{j1}(\vx),-\Upsilon_{j2}(\vx))$ in a neighbourhood of the curve $C_j$. Defining $\Upsilon(\vx)$ this way ensures that $\Upsilon(\vx)$ is lower $C^1$~\cite{rockafellar2009variational,rockafellar1970convex}. The subgradient of $\Upsilon(\vx)$ along $C_j(\epsilon)$ is given by
\begin{eqnarray}\label{eq:convexity_function}
\partial\Upsilon(\vx) = \{\gamma\nabla\Upsilon_{j1}(\vx) + (1-\gamma)\Upsilon_{j2}(\vx)\ |\ \gamma\in[0,1]\}.
\end{eqnarray}
Using the continuity of $\Upsilon$, we can apply the chain rule of gradients~\cite[Theorem 10.6]{rockafellar2009variational} to get
\begin{eqnarray}\label{eq:chain_rule_subgradient}
\partial(\Upsilon\circ\vx)(t)\subset\{\z\cdot\dot{\vx}(t)\,|\,\z\in\partial\Upsilon(\vx(t))\}.
\end{eqnarray}
Note that Proposition~\ref{prop:invariant} shows that $\mathcal{M}_{\epsilon}$ is invariant, i.e., the vector field along its boundary points towards the interior of $\mathcal{M}_{\epsilon}$. Consider a compact neighbourhood $\mathcal{K}$ of the curve $C_j$. Since $\mathcal{K}$ is compact, there is a $\delta>0$ such that $\dot{\vx}\cdot\nabla\Upsilon_{j1}<-\delta$ and $\dot{\vx}\cdot\nabla\Upsilon_{j2}<-\delta$ on $\mathcal{K}$. From~(\ref{eq:convexity_function}), we get that there exists a $\delta>0$ such that $\z\cdot \dot{\vx}(t)<-\delta<0$ for all $\z\in\partial\Upsilon(\vx(t))$ in $\mathcal{K}$. Using~(\ref{eq:chain_rule_subgradient}), we get
\begin{eqnarray}
\sup\limits_{t\geq 0}\partial (\Upsilon\circ\vx)(t)<-\delta.
\end{eqnarray}
Since $\Upsilon$ is lower $C^1$, one can apply the mean value theorem~\cite[Theorem 10.48]{rockafellar2009variational} to $\Upsilon\circ\vx(t)$ to get that there is a $\tau\in[0,t]$
\begin{eqnarray}
\displaystyle\Upsilon(\vx(t))-\Upsilon(\vx(0)) = t\alpha_t  \text{ for some } \alpha_t\in \partial (\Upsilon\circ\vx)(\tau).
\end{eqnarray} 
Since $\alpha_t < \delta$, this implies that on $\mathcal{K}$, we have
\begin{eqnarray}
\Upsilon(\vx(t)) < \Upsilon(\vx(0)) - \delta t
\end{eqnarray}
for all $t>0$. This contradicts the fact that $\Upsilon(\vx(t))\in[\frac{1}{\epsilon_0},\frac{1}{\epsilon_1}]$ for all $t> 0$.\\

Case II: Consider points on the boundary of $\mathcal{M}_{\epsilon}$ that are end points of trajectories. In this case, the angle made by $\mathcal{M}_{\epsilon}$ along $C_j$ can be equal to or different from $\pi$ depending on whether the eigenvalues of the Jacobian are equal or not. From Proposition~\ref{prop:finite_equal_eigenvalues}, we know that the set of points when the eigenvalues of the Jacobian are equal is finite. Let $(\epsilon_1,\epsilon_2,....,\epsilon_k)$ be the set such that for each $\epsilon_i$ in $(\epsilon_1,\epsilon_2,....,\epsilon_k)$, the boundary of $\mathcal{M}_{{\epsilon}_i}$ contains end points of trajectories where the Jacobian has equal eigenvalues. For each such $\epsilon_i$, contruct a small enough annular region around $\mathcal{M}_{{\epsilon}_i}$. Note that since the annular region is a compact set, by continuity there exists a $\delta_0$ such that $\dot{\vx}\cdot\Delta\Upsilon<-\delta_0$. Between the annular regions, the function $\Upsilon(\vx)$ is $C^1$. Therefore there exists a $\delta_1$ such that $\dot{\vx}\cdot\Delta\Upsilon<-\delta_1$. Therefore, we have $\dot{\vx}\cdot\Delta\Upsilon< \min(-\delta_0,-\delta_1)$. We can now repeat the procedure as in Case I to get our desired conclusion. The only case that remains to be resolved when we have distinct eigenvalues is when we start along the faster eigen direction. However, this case does not occur due to Proposition~\ref{prop:fast_eigendirection}.

\end{proof}

\begin{theorem}
Consider a dynamical system $\GG_{\epsilon}^{\rm{variable\mbox{-}k}}$. Then $\mathcal{M}_{\epsilon}$ is the minimal globally attracting region for $\GG_{\epsilon}^{\rm{variable\mbox{-}k}}$.
\end{theorem}

\begin{proof}
Theorem~\ref{thm:globally_attracting} shows that $\mathcal{M}_{\epsilon}$ is a globally attracting region. We now show that it is the minimal globally attracting region, i.e., it is contained in every globally attracting region. Towards this, we will show that each point in $\mathcal{M}_{\epsilon}$ lies in the omega-limit set of some trajectory of $\GG_{\epsilon}^{\rm{variable\mbox{-}k}}$. In particular, let $Q\in\mathcal{M}_{\epsilon}$. We will show that $Q$ lies in the omega-limit point of some trajectory of $\GG_{\epsilon}^{\rm{variable\mbox{-}k}}$. Consider $P_2\in\mathcal{M}_{\epsilon}$ such that $P_2\neq Q$. From Proposition~\ref{prop:path_invariant}, we have $P_1\leadsto P_2$ for any $P_1\in\mathbb{R}^2_{>0}$. Choose some $\eta_1 >0$. Then there exists a time $t_1$ and trajectory $\vx(t)$ with $\vx(0)=P_2$ such that $||\vx(t_1) - Q||<\eta_1$. Choose $\eta'_1>0$. Using Proposition~\ref{prop:path_invariant} again, we get that there exists a time $t'_1 > t_1$ and a trajectory starting at $\vx(t_1)$ such that $||\vx(t'_1) - P_2||< \eta'_1$. Now choose $\eta_2 >0$ such that $\eta_2 < \eta_1$. Using Proposition~\ref{prop:path_invariant} again, we get that there exists a time $t_2 > t'_1$ and trajectory starting at $\vx(t'_1)$ such that $||\vx(t_2) - Q||<\eta_2$. Repeating this between the points $P_2$ and $Q$ generates a trajectory $\vx(t)$ and a sequence of times $t_1 < t_2 < ... < t_k$ such that $\displaystyle\lim_{k\to\infty}\vx(t_k) = Q$, implying that $Q$ lies in the omega-limit of this trajectory.

\end{proof}

\section{Discussion}

In this paper, we have constructed minimal invariant regions and minimal globally attracting regions for variable-$k$ dynamical systems generated by networks possessing two reversible reactions. In this special case, the minimal invariant region coincides with the minimal globally attracting region. Of course, these regions  are also invariant and globally attracting regions for the corresponding fixed-$k$ mass-action systems.

In previous work~\cite{ding2020minimal} we have constructed minimal invariant regions and minimal globally attracting regions for general toric differential inclusions~\cite{craciun2019quasi,craciun2015toric} in two dimensions. Therefore, since large classes of mass-action systems can be {\em embedded}~\cite{craciun2020endotactic,craciun2019polynomial} into toric differential inclusions, this provides some invariant regions and some globally attracting regions (but not necessarily minimal ones) for many variable-$k$ and fixed-$k$ mass-action systems with any number of reactions, even if they are not reversible, as long as they can be embedded into toric differential inclusions. In particular, this applies to all weakly reversible and to all endotactic networks in two dimensions.

We have only considered here variable-$k$ dynamical systems with {\em two} reversible reactions; this is the simplest nontrivial case for this class of problems, and we regard the results obtained here as a proof-of-concept for future work in this area. Numerical simulations suggest that the analysis of the more general case with arbitrary number of reversible reactions can be significantly more complicated. Similarly, numerical simulations for the construction of minimal invariant regions and minimal globally attracting regions for fixed-$k$ dynamical systems suggest that this problem might also be quite difficult, in general. We think that these are very interesting avenues for future work.

\bibliographystyle{amsplain}
\bibliography{Bibliography}

\end{document}